\documentclass[]{interact}

\usepackage{amsmath, amsfonts, amssymb, amsthm}

\usepackage{epstopdf}
\usepackage[caption=false]{subfig}

\usepackage[numbers,sort&compress]{natbib}
\bibpunct[, ]{[}{]}{,}{n}{,}{,}

\theoremstyle{plain}
\newtheorem{theorem}{Theorem}[section]

\newtheorem{corollary}[theorem]{Corollary}

\theoremstyle{definition}
\newtheorem{definition}[theorem]{Definition}
\newtheorem{example}[theorem]{Example}

\theoremstyle{remark}
\newtheorem{remark}{Remark}

\newcommand{\bbN}{\mathbb{N}}

\newcommand{\bbR}{\mathbb{R}}

\begin{document}
\title{Green's Function for Higher-Order Boundary Value Problems Involving a Nabla Caputo Fractional Operator}

\author{
\name{Kevin Ahrendt\thanks{CONTACT Kevin Ahrendt. Email: kahrendt@mines.edu} and Cameron Kissler}
\affil{Colorado School of Mines, Department of Applied Mathematics and Statistics, Golden, CO 80401, USA}
}

\maketitle

\begin{abstract}
We consider the discrete, fractional operator $\left(L_a^\nu x\right) (t) := \nabla [p(t) \nabla_{a^*}^\nu x(t)] + q(t) x(t-1)$ involving the nabla Caputo fractional difference, which can be thought of as an analogue to the self-adjoint differential operator. We show that solutions to difference equations involving this operator have expected properties, such as the form of solutions to homogeneous and nonhomogeneous equations. We also give a variation of constants formula via a Cauchy function in order to solve initial value problems involving $L_a^\nu$. We also consider boundary value problems of any fractional order involving $L_a^\nu$. We solve these BVPs by giving a definition of a Green's function along with a corresponding Green's Theorem. Finally, we consider a (2,1) conjugate BVP as a special case of the more general Green's function definition.
\end{abstract}


\begin{amscode}
39A10, 39A70
\end{amscode}

In this paper, we consider boundary value problems with the discrete, fractional operator $\left(L_a^\nu x\right)(t) := \nabla [p(t) \nabla_{a^*}^\nu x(t)] + q(t) x(t-1)$ which involves the Caputo fractional difference. This particular operator, when $0 < \nu < 1$, can be thought of as a discrete, fractional analogue to the self-adjoint differential operator. We extend results for this operator in the case of $\nu > 1$, building off the work in \cite{ahrendt2017}.

Recently, nabla fractional calculus has been investigated and developed in detail by Abdeljawad, Atici and Eloe, and Hein et al.\cite{Hein,Abdeljawad,aticiEloe}, among others. The Caputo fractional difference, in particular, has been of interest to Ahrendt et al., Anastassiou, and Jonnalagadda \cite{anas,ahrendt2016,Jonnalagadda}. Goodrich and Peterson \cite{goodrichPeterson} give a self-contained introduction to the nabla fractional calculus involving both the Riemann-Liouville and Caputo fractional differences. Boundary value problems involving the nabla Caputo fractional difference have been studied by Ikram and St. Goar \cite{ikram, StGoar}.

In the first section, we introduce the fractional calculus background required. In the second section, we investigate the $L_a^\nu$ operator in the context of initial value problems. In the third section, we consider boundary value problems involving this operator with $N$ boundary conditions at the left endpoint and $1$ boundary condition at the right endpoint. The final section gives an example of a specific Green's function for a (2,1) conjugate boundary value problem.

\section{Preliminary definitions and theorems}
\begin{definition}
If $a\in \bbR$, then the set $\bbN_a$ is given by $\{a,a+1,a+2,\hdots\}$. Furthermore, if $b \in \bbN_a$, then $\bbN_a^b$ is given by $\{a,a+1, \ldots, b-1,b\}$.
\end{definition}


\begin{definition}\cite{goodrichPeterson}
Let $f:\bbN_a \to \bbR$.  Then, the \textit{nabla difference of $f$} is defined by 
\[
(\nabla f)(t) := f(t) - f(\rho(t))= f(t) - f(t-1),
\]
for $t\in \bbN_{a+1}$, where $\rho(t) :=t-1$ is the backwards jump operator.  For convenience, we will use the notation $\nabla f(t) := (\nabla f)(t)$. For $N \in \bbN_2$, we have that the \textit{$N$-th order difference} is recursively defined by
\[
\nabla^N f(t) := \nabla (\nabla^{N-1} f(t)),
\]
for $t\in \bbN_{a+N}$.   Finally, we take by convention $\nabla^0 f(t) = f(t)$.
\end{definition}


\begin{definition}\cite{goodrichPeterson}
Let $f:\bbN_a \to \bbR$ and let $c,d\in \bbN_a$.  Then, the \textit{definite nabla integral of $f$ from $c$ to $d$} is defined by
\[
\int_c^d f(s)\nabla s := \begin{cases}
					\sum_{s=c+1}^d f(s), & c<d, \\
					0, & d\leq c.
				\end{cases}
\]
\end{definition}

\begin{theorem}[Nabla Leibniz Rules]\cite{goodrichPeterson}\label{leibnizFormula}
Assume $f: \bbN_a \times \bbN_{a+1} \to \bbR$. Then,
\begin{enumerate}
\item $\nabla \left( \int_{a}^t f(t, s) \nabla s \right) = \int_a^t \nabla_t f(t, s) \nabla s + f(\rho(t),t)$;
\item $\nabla \left( \int_{a}^t f(t,s) \nabla s \right) = \int_a^{t-1} \nabla_t f(t,s) \nabla s + f(t,t)$.
\end{enumerate}
\end{theorem}

We desire an analogue to the continuous power rule. For example, $\nabla t^2 = t^2-(t-1)^2 = 2t-1$, so the power rule involving a nabla difference does not hold with regular exponents. The following rising function replaces the exponent to allow a power rule in the nabla difference case.

\begin{definition}\cite{goodrichPeterson}
Let $t\in \bbR$ and $\nu \in \bbR$, the \textit{generalized rising function} is defined by
\[
t^{\overline \nu} := \frac{\Gamma(t+\nu)}{\Gamma(t)},
\]
for $t$ and $\nu$ such that $t+\nu \not \in \{\ldots,-2,-1,0\}$. If $t$ is a non-positive integer and $t + \nu$ is not a non-positive integer, then we take by convention $t^{\overline \nu} = 0$.  
\end{definition}

%

To extend the whole-order difference and integral to fractional values, we define nabla fractional Taylor monomials.

\begin{definition}\cite{goodrichPeterson}
Let $\nu \in \bbR$. The \textit{$\nu$-th order nabla fractional Taylor monomial}, denoted $H_\nu(t,a)$, is given by
\[
H_\nu(t,a) := \frac{(t-a)^{\overline{\nu}}}{\Gamma(\nu+1)},
\]
whenever the right-hand side of the equation is sensible. By convention, we take $H_0(t,a) = 1$.
\end{definition}

%
%

The following definition extends the nabla integral to non-integer orders.

\begin{definition}\cite{goodrichPeterson}
Let $f:\bbN_{a+1}\rightarrow\mathbb{R}$ and $\nu > 0$. The \textit{$\nu$-th order nabla fractional integral of f, based at a,} is defined as 
	\[
	\nabla_a^{-\nu}f(t):=\int_{a}^t H_{\nu-1}(t,\rho(s))f(s) \nabla s,
	\]
for $t\in\bbN_a$. By the convention regarding nabla integrals, we take $\nabla_a^{-\nu}f(a) = 0$.
\end{definition}

\begin{theorem}[Composition Rules]\cite{goodrichPeterson} \label{compositionRule}
Suppose $f: \bbN_{a} \rightarrow \mathbb{R}$. Let $\mu> 0$ and $N \in \bbN_1$ be given. The following hold:
\begin{enumerate}
\item $\nabla^N \nabla_a^{-\mu} f(t) = \nabla_a^{N-\mu} f(t)$, \quad for $t \in \bbN_a$;
\item $\nabla_a^{-\mu} \nabla_a^{\mu} f(t) = f(t)$, \quad for $t \in \bbN_a$;
\end{enumerate}
where $\nabla_a^\mu$ is the Riemann-Liouville fractional difference.
\end{theorem}
%

We now use the definition of the nabla fractional integral and the definition of the whole order difference to define the Caputo fractional difference.

\begin{definition}\cite{goodrichPeterson}
Let $f:\bbN_{a-N+1}\rightarrow\mathbb{R}$, $\nu>0$, and $N=\lceil\nu\rceil$.  The \textit{$\nu$-th order Caputo nabla fractional difference} is defined by
	\[
	\nabla_{a*}^{\nu}f(t):=\nabla_a^{-(N-\nu)}(\nabla^Nf(t)),
	\]
for $t\in\bbN_{a}$. Note that the Caputo difference operator is a linear operator, and we consider a difference equation involving just the Caputo operator to behave like it is of order $N$.
\end{definition}

\section{Fractional Initial Value Problems for a Particular Operator}
This section introduces basic theory regarding a particular discrete, fractional operator, as well as how it applies to initial value problems. It builds on Ahrendt et al. \cite{ahrendt2016} by considering cases when $\nu >1$, while also correcting an off by one error present in that paper.
\begin{definition}
Let $\nu \in \bbR$ be given such that $N-1<\nu<N$ for some $N \in \mathbb{N}_1$. We define the fractional operator $L_a^\nu$ based at $a$ with Caputo order $\nu$ by
\[
(L_a^\nu x)(t):= \nabla[p(t)\nabla_{a^*}^\nu x(t)]+ q(t)x(t-1), \quad t \in \mathbb{N}_{a+N+1},
\]
where $p: \bbN_{a+N} \to (0,\infty)$ and $q: \bbN_{a+N+1}\to\bbR$. Note that $\left(L_a^\nu x\right) (t)$ is defined for $t \in \bbN_{a+N+1}$, but $x(t)$ must be defined for $t \in \bbN_a$. Also, $L_a^\nu$ is a linear operator. Finally, the order of the operator $L_a^\nu$ behaves like it is $N+1$, as there is a whole order difference operating on the Caputo operator of order $\nu$.
\end{definition}
\begin{remark}
The $L_a^\nu$ operator takes the form of a discrete, fractional version of the self-adjoint differential equation when $0 < \nu < 1$. This case is considered in detail in \cite{ahrendt2017}.
\end{remark}

The following theorem is proven by expanding the operators in $L_a^\nu$ and solving for $x(t)$ in terms of the known functions $p(t)$, $q(t)$, and $x(s)$, for $s \in \bbN_a^{t-1}$.
\begin{theorem}[Existence and Uniqueness for IVPs]\label{existenceUniquenessThm}
Let $\nu \in \bbR$ such that $N-1 < \nu < N$ for some $N \in \bbN_1$ and $h: \bbN_{a+N+1} \to \bbR$. If $p(t) \neq \frac{p(a+1)}{H_{-1-\nu}(t+1,a)-H_{N-1-\nu}(t+1,a)}$ for some $ t\in \bbN_{a+N+1}$, then the initial value problem
\[
\begin{cases}
L_a^\nu x(t) = \nabla[p(t)\nabla_{a^*}^\nu x(t)]+q(t) x(t-1) = 0, & t \in \bbN_{a+N+1}, \\
\nabla^i x(a+i) = A_i, \quad i \in \bbN_0^{N},
\end{cases}
\]
where $A_i \in \bbR$ for $i \in \bbN_0^{N}$, has a unique solution $x: \bbN_a \to \bbR$.
\end{theorem}

This next theorem and corollary are proven in a standard way, relying on the linearity of the $L_a^\nu$ operator and Theorem \ref{existenceUniquenessThm}.
\begin{theorem}[General Solution of the Homogeneous Equation]\label{generalSolutionThm}
Let $\nu \in \bbR$ such that $N-1 < \nu < N$ for some $N \in \bbN_1$ be given. Suppose $x_0$, $x_1$, \ldots, $x_{N}$ are $N+1$ linearly independent solutions to $L_a^\nu x(t) =0$. The general solution to $L_a^\nu x(t) = 0$ is given by
\[
x(t) = \sum_{i=0}^N c_i x_i (t), \quad t\in \bbN_a,
\]
where $c_0,c_1,\ldots c_N \in \bbR$ are arbitrary constants.
\end{theorem}
%
%

\begin{corollary}[General Solution of the Nonhomogeneous Equation]\label{nonHomogeneousGeneralSolutionThm}
Let $\nu \in \bbR$ such that $N-1 < \nu < N$ for some $N \in \bbN_1$ be given. Also, let $x_0$, $x_1$, \ldots, $x_{N}$ $: \bbN_a \to \bbR$ be $N+1$ linearly independent solutions to $L_a^\nu x(t) = 0$. If $x_p : \bbN_a \to \bbR$ is a particular solution to $L_a^\nu x(t) = h(t)$ for some $h : \bbN_{a+N+1} \to \bbR$, then the general solution to $L_a^\nu x(t) = h(t)$, for $t \in \bbN_{a+N+1}$, is given by
\[
x(t) = x_p(t)+ \sum_{i=0}^{N} c_ix_i(t),
\]
for arbitrary constants $c_i \in \bbR$ for $i \in \bbN_0^{N}$.
\end{corollary}

In order to derive a variations of constants formula, we define the following Cauchy function.

\begin{definition}[Cauchy Function]\label{cauchyDefn}
The Cauchy function for $L_a^\nu x(t) = 0$ 
 is the function $x(t,s)$ where $x: \bbN_a \times \bbN_{a+N+1} \to \bbR$ and, for any fixed $s \in \bbN_{a+N+1}$,  satisfies the initial value problem
\[
\begin{cases}
L_{\rho(s)}^\nu x(t,s) = 0, \quad t \in \bbN_{s+N}, \\
\nabla^i x(\rho(s),s) = 0, \quad i \in \bbN_0^{N-1}, \\
\nabla^{N} x(s,s) = \frac{1}{p(s)}.
\end{cases}
\]
\end{definition}

\begin{theorem}[Variation of Constants]\label{variationOfConstantsThm}
Let $\nu \in \bbR$ such that $N-1< \nu < N$ for some $N \in \bbN_1$ and $h : \bbN_{a+N+1} \to \bbR$. Then, the solution to the initial value problem
\[
\begin{cases}
L_a^\nu x(t) = h(t), \quad t \in \bbN_{a+N+1}, \\
\nabla^i x(a+i) = 0, \quad i \in \bbN_0^{N},
\end{cases}
\]
is given by
\[
x(t) = \int_{a+N}^t x(t,s) h(s) \nabla s,
\]
where $x(t,s)$ is the Cauchy function given in Definition \ref{cauchyDefn}.
\end{theorem}
\begin{proof}
First note that, by convention, $x(t) =\int_{a+N}^t x(t,s) h(s) \nabla s$ is $0$ for $t \in \bbN_a^{a+N}$. This equivalently shows that $\nabla^i x(a+i) = 0$ for $i \in \bbN_0^{N}$. Hence, $x(t)$ satisfies the initial conditions.

Consider
\begin{align*}
\nabla\left[ p(t) \nabla_{a^*}^\nu x(t) \right]  &= \nabla\left[p(t) \nabla_{a^*}^\nu \int_{a+N}^t x(t,s)h(s) \nabla s\right]  \\
=& \nabla \left[ p(t) \nabla_a^{-(N-\nu)} \nabla^N \int_{a+N}^t x(t,s) h(s) \nabla s \right] \\
\stackrel{\text{Thm } \ref{leibnizFormula}}{=} & \nabla \left[ p(t) \nabla_a^{-(N-\nu)} \nabla^{N-1}\left( \int_{a+N}^t \nabla_t x(t,s) h(s) \nabla s + x(\rho(t),t) h(s) \right) \right] \\
=& \nabla \left[ p(t) \nabla_a^{-(N-\nu)} \nabla^{N-1}\left( \int_{a+N}^t \nabla_t x(t,s) h(s) \nabla s + 0\cdot h(s) \right) \right] \\
\stackrel{\text{Thm } \ref{leibnizFormula}}{=} &  \nabla \left[ p(t) \nabla_a^{-(N-\nu)} \nabla^{N-2}\left( \int_{a+N}^t \nabla^2_t x(t,s) h(s) \nabla s + \nabla x(\rho(t),t) h(s) \right) \right]\\
=&  \nabla \left[ p(t) \nabla_a^{-(N-\nu)} \nabla^{N-2}\left( \int_{a+N}^t \nabla^2_t x(t,s) h(s) \nabla s + 0 \cdot h(s) \right) \right],
\end{align*}
using the fact that $\nabla^i x(\rho(t),t)=0 $ for $i \in \bbN_0^{N-1}$ from the definition of the Cauchy function. Applying Leibniz's formula $N-2$ more times yields
\begin{align*}
\nabla\left[ p(t) \nabla_{a^*}^\nu x(t) \right] &=  \nabla \left[ p(t) \nabla_a^{-(N-\nu)} \int_{a+N}^t \nabla^{N}_t x(t,s) h(s) \nabla s \right] \\
&= \nabla\left[p(t) \int_a^t H_{N-\nu-1}(t, \rho(\tau)) \left(\int_{a+N}^\tau \nabla_\tau^N x(\tau,s) h(s) \nabla s\right) \nabla \tau \right] \\
&= \nabla \left[ p(t) \sum_{\tau = a+1}^t \sum_{s= a+N+1}^\tau H_{N-\nu-1}(t,\rho(\tau)) \nabla_\tau^N x(\tau,s) h(s) \right].
\end{align*}
Interchanging the order of summation, we have
\begin{align*}
\nabla\left[ p(t) \nabla_{a^*}^\nu x(t) \right] &= \nabla \left[ p(t) \sum_{s=a+N+1}^t h(s) \sum_{\tau=s}^t H_{N-\nu-1}(t, \rho(\tau)) \nabla_\tau^N x(\tau,s) \right] \\
&= \nabla \left[\int_{a+N}^t p(t) h(s) \left(\int_{\rho(s)}^t H_{N-\nu-1}(t, \rho(\tau)) \nabla_\tau^N x(\tau,s) \nabla \tau\right) \nabla s \right].
\end{align*}
Applying Theorem \ref{leibnizFormula} and using the last initial condition for the Cauchy function, we get
\begin{align*}
\nabla\left[ p(t) \nabla_{a^*}^\nu x(t) \right] &= \int_{a+N}^{t-1} h(s) \nabla_t \left[p(t) \int_{\rho(s)}^t H_{N-\nu-1}(t, \rho(\tau)) \nabla_\tau^N x(\tau,s) \nabla \tau \right] \nabla s \\
&\quad\quad+ h(t) p(t) H_{N-\nu-1}(t, \rho(t)) \nabla_t^N x(t,t) \\
&=\left( \int_{a+N}^{t-1} h(s) \nabla_t \left[ p(t) \nabla_{\rho(s)^*}^\nu x(t,s) \right] \nabla s\right) + h(t) p(t) \frac{1}{p(t)}.
\end{align*}
Finally, consider
\begin{align*}
L_a^\nu x(t) &= \nabla[p(t) \nabla_{a^*}^\nu x(t)] + q(t) x(t-1) \\
&=h(t) + \int_{a+N}^{t-1} h(s) \nabla_t \left[ p(t) \nabla_{\rho(s)^*}^\nu x(t,s) \right] \nabla s + q(t) \int_{a+N}^{t-1} x(t-1,s) h(s) \nabla s \\
&=h(t) + \int_{a+N}^{t-1} h(s) L_{\rho(s)}x(t,s) \nabla s  \\
&=h(t).
\end{align*}
Therefore, $x(t) = \int_{a+N}^t x(t,s) h(s) \nabla s$ solves the above initial value problem.
\end{proof}

\begin{example} \label{cauchyFnctnWithP}
The Cauchy function for $\nabla[ p(t) \nabla_{a^*}^\nu x(t) ] =0$ is given by $x(t,s) = \nabla_{\rho(s)}^{-\nu} \frac{1}{p(t)}$.
\end{example}
\begin{proof}
We will verify that $x(t,s) = \nabla_{\rho(s)}^{-\nu} \frac{1}{p(t)}$ satisfies the IVP in Definition \ref{cauchyDefn}. Let $s \in \bbN_{a+N+1}$ be fixed and consider
\begin{align*}
\nabla\left[p(t) \nabla_{\rho(s)^*}^\nu x(t,s) \right] &= \nabla\left[p(t) \nabla_{\rho(s)}^{-(N-\nu)}\left( \nabla^N \left(\nabla_{\rho(s)}^{-\nu} \frac{1}{p(t)}\right)\right) \right] \\
&\stackrel{\text{Thm } \ref{compositionRule}}{=}  \nabla \left[ p(t) \nabla_{\rho(s)}^{-(N-\nu)}\left( \nabla_{\rho(s)}^{N-\nu} \frac{1}{p(t)}\right) \right] \\
&\stackrel{\text{Thm } \ref{compositionRule}}{=}  \nabla \left[p(t) \frac{1}{p(t)}\right]\\
&= 0,
\end{align*}
so the equation is satisfied. For $i \in \bbN_0^{N}$,
\[
\nabla^i x(t,s) = \nabla^i \nabla_{\rho(s)}^{-\nu} \frac{1}{p(t)}\stackrel{\text{Thm } \ref{compositionRule}}{=} \nabla_{\rho(s)}^{-(\nu-i)} \frac{1}{p(t)} = \int_{\rho(s)}^t H_{\nu-i-1}(t,\rho(\tau))\frac{1}{p(\tau)} \nabla \tau.
\]
Note, when $i \in \bbN_0^{N-1}$ and $t =\rho(s)$, we take by convention that the above integral is $0$. When $i = N$ and $t=s$, we have
\begin{align*}
\left. \nabla^{N} x(t,s)\right|_{t=s} &=\left. \int_{\rho(s)}^t H_{\nu-N-1}\left(t,\rho\left(\tau\right)\right)\frac{1}{p(\tau)} \nabla \tau \right|_{t=s}\\
&= \sum_{\tau = s}^s H_{\nu-N-1}\left(s, \rho\left(\tau\right)\right) \frac{1}{p(\tau)} \\
&= H_{\nu-N}(s,\rho(s)) \frac{1}{p(s)} \\
&= \frac{1}{p(s)},
\end{align*}
thus $x(t,s)$ satisfies the IVP in Definition \ref{cauchyDefn}. Hence, $x(t,s) = \nabla_{\rho(s)}^{-\nu} \frac{1}{p(t)}$ is the Cauchy function for $\nabla[ p(t) \nabla_{a^*}^\nu x(t) ] =0$.

\end{proof}

\begin{example}\label{mostBasicCauchyFnctn}
The Cauchy function for $\nabla \nabla_{a^*}^\nu x(t) = 0$ is given by $x(t,s) = H_\nu(t,\rho(s))$.
\end{example}
\begin{proof}
This is a specific case of Example \ref{cauchyFnctnWithP} where $p(t) \equiv 1$. Hence,
\begin{align*}
x(t,s) &= \nabla_{\rho(s)}^{-\nu} 1 \\
&= \int_{\rho(s)}^t H_{\nu-1}(t, \rho(\tau)) \nabla \tau \\
&=H_\nu(t,\rho(s)),
\end{align*}
using a property of Taylor monomials given in \cite[p. 186]{goodrichPeterson}.
\end{proof}

\section{Boundary Value Problems Involving $L_a^\nu$}

In this section, we will explore boundary value problems involving $L_a^\nu$ with $N$ boundary conditions at the left endpoint, and $1$ boundary condition at the right endpoint. The general boundary conditions used here can account for several specific boundary conditions. For the case $0 < \nu <1$, \cite{ahrendt2017} already considered conjugate boundary conditions, and \cite{StGoar} already considered right-focal boundary conditions. For $\nu >1$, one could consider various other standard boundary conditions, and as an example, we will look at  (2,1) conjugate BVP in the last section of the paper.

In particular, we are interested in the homogeneous self-adjoint boundary value problem
\begin{equation} \label{homogeneousBVP}
\begin{cases}
L_a^\nu x(t) = 0, \quad t \in \bbN_{a+N+1}^b, \\
\sum_{j=0}^N \alpha_{ij} \nabla^j x(a+j) =0, \quad i \in \bbN_0^{N-1}, \\
\sum_{j=0}^N \beta_j \nabla^j x(b) =0, 
\end{cases}
\end{equation}
and the corresponding nonhomogeneous boundary value problem
\begin{equation}\label{nonHomogeneousBVP}
\begin{cases}
L_a^\nu x(t) = h(t), \quad t \in \bbN_{a+N+1}^b, \\
\sum_{j=0}^N \alpha_{ij} \nabla^j x(a+j) =A_i, \quad i \in \bbN_0^{N-1}, \\
\sum_{j=0}^N \beta_{j} \nabla^j x(b) =B,
\end{cases}
\end{equation}
where $b-a \in \bbN_{N+1}$; $\alpha_{ij} \in \bbR$ for $j \in \bbN_0^N$ such that $\sum_{j=0}^N \alpha_{ij}^2 > 0$ for each $i \in \bbN_0^{N-1}$; $\beta_j \in \bbR$ for $j \in \bbN_0^N$ such that $\sum_{j=0}^N \beta^2_j > 0$; $\left\{\langle \alpha_{i0}, \alpha_{i1}, \ldots, \alpha_{iN} \rangle\right\}_{i=0}^{N-1}$ are linearly independent; $A_i, B \in \bbR$ for $i \in \bbN_0^{N-1}$; and $h: \bbN_{a+N+1}^b \to \bbR$.

\begin{theorem}
If \eqref{homogeneousBVP} has only the trivial solution, then \eqref{nonHomogeneousBVP} has a unique solution.
\begin{proof}
Let $x_k : \bbN_a \to \bbR$, for $k\in \bbN_0^{N}$ be $N+1$ linearly independent solutions to $L_a^\nu x(t) =0$. This implies, by Theorem \ref{generalSolutionThm}, a general solution is given by $x(t) = \sum_{k=0}^N c_k x_k(t)$, where $c_k \in \bbR$ for $k \in \bbN_0^N$ are arbitrary constants. For convenience, let
\begin{align*}
x_k^i &:= \sum_{j=0}^N \alpha_{ij} \nabla^j x_k(a+j), \quad \text{for } i \in \bbN_0^{N-1},\\
x_k^N &:= \sum_{j=0}^N \beta_j \nabla^j x_k(b).
\end{align*}

Note that $x(t)$ satisfies the boundary conditions in \eqref{homogeneousBVP} if and only if it satisfies the following system of equations:
\[
\begin{cases}
\sum_{k=0}^N c_k x_k^0 &= 0, \\
\sum_{k=0}^N c_k x_k^1 &= 0, \\
\quad\quad\vdots &=\vdots \\
\sum_{k=0}^N c_k x_k^N &=0,
\end{cases}
\]
if and only if the vector equation
\[
\underbrace{
\left(
\begin{matrix}
x_0^0 & x_1^0 & \cdots & x_N^0 \\
x_0^1 & x_1^1 & \cdots & x_N^1 \\
\vdots & \vdots & \ddots & \vdots \\
x_0^N & x_1^N & \cdots & x_N^N
\end{matrix}
\right)
}_{=:D}
\left(
\begin{matrix}
c_0 \\
c_1 \\
\vdots \\
c_N
\end{matrix}
\right)
=
\left(
\begin{matrix}
0 \\
0\\
\vdots \\
0
\end{matrix}
\right).
\]

We have that $x(t)$ is the trivial solution if and only if $c_0=c_1=\cdots=c_N=0$ which happens if and only if $\det(D)\neq0$, noting that $\left\{\langle \alpha_{i0}, \alpha_{i1}, \ldots, \alpha_{iN} \rangle\right\}_{i=0}^{N-1}$ are linearly independent.

By Corollary \ref{nonHomogeneousGeneralSolutionThm}, a general solution to \eqref{nonHomogeneousBVP} is given by $x(t) = x_p(t)+ \sum_{k=0}^N a_k x_k(t)$, where $a_k \in \bbR$, for $k \in \bbN_0^N$, are arbitrary constants and $x_p : \bbN_a\to \bbR$ is a particular solution to $L_a^\nu x(t) =h(t)$. To satisfy the boundary conditions in \eqref{nonHomogeneousBVP}, $x(t)$ must satisfy the vector equation
\begin{equation} \label{nonHomogeneousVector}
D \left(
\begin{matrix}
a_0\\
a_1\\
\vdots\\
a_N
\end{matrix}
\right)
=
\left(
\begin{matrix}
A_0 - \sum_{j=0}^N \alpha_{0j} \nabla^j x_p(a+j) \\
A_1 - \sum_{j=0}^N \alpha_{1j}\nabla^j x_p(a+j) \\
\vdots \\
A_{N-1} - \sum_{j=0}^N \alpha_{(N-1)j} \nabla^j x_p(a+j) \\
B - \sum_{j=0}^N \beta_j\nabla^j x_p(b).
\end{matrix}
\right)
\end{equation}

Since the homogeneous BVP \eqref{homogeneousBVP} has only the trivial solution, we have $\det(D) \neq 0$, so \eqref{nonHomogeneousVector} is satisfied with unique values for $a_i$ for $i \in \bbN_0^N$. Hence, $x(t) = x_p(t)+ \sum_{k=0}^N a_k x_k(t)$ uniquely solves the nonhomogeneous BVP \eqref{nonHomogeneousBVP}.

\end{proof}
\end{theorem}

\begin{definition}[Green's Function] \label{greensFnctnDefn}
Assume that \eqref{homogeneousBVP} has only the trivial solution. The \textit{Green's function} $G(t,s)$, where $G:\mathbb{N}_a^b\times\mathbb{N}_{a+N+1}^b\to\mathbb{R}$, for the homogeneous boundary value problem \eqref{homogeneousBVP} is given by
\[
G(t,s):=\begin{cases} u(t,s), &t\in\bbN_a^{b-N} \textrm{ and } s\in\bbN_{\max\{t+1,a+N+1\}}^b, \\ 
v(t,s),&t\in\mathbb{N}_{a+N}^b \textrm{ and } s\in\mathbb{N}_{a+N+1}^{\min\{t+1,b\}},
\end{cases}
\]
where, for each fixed $s\in\mathbb{N}_{a+N+1}^b$, $u(t,s)$ solves the boundary value problem
\begin{equation} \label{uGreensBVP}
\begin{cases} 
L_a^{\nu}u(t,s)=0, \quad t\in\mathbb{N}_{a+N+1}^b, \\
\sum_{j=0}^N \alpha_{ij} \nabla_t^j u(t,s)|_{t=a+j}=0, \quad i\in\mathbb{N}_0^{N-1}, \\
\sum_{j=0}^N \beta_{j} \nabla_t^j u(t,s)|_{t=b} = -\sum_{j=0}^N \beta_{j} \nabla_t^j x(t,s)|_{t=b},
\end{cases}
\end{equation}
and $v(t,s):= u(t,s) + x(t,s)$, where $x(t,s)$ is the Cauchy function for $L_a^{\nu}x(t)=0$.
\end{definition}

\begin{theorem}[Green's Function Theorem] \label{greensFnctnThm}
If \eqref{homogeneousBVP} has only the trivial solution, then the unique solution to \eqref{nonHomogeneousBVP}, with $A_i = B =0$ for $i \in \bbN_0^{N-1}$, is given by
\[
x(t) = \int_{a+N}^b G(t,s)h(s) \nabla s,
\]
where $G(t,s)$ is the Green's function for the homogeneous boundary value problem \eqref{homogeneousBVP}.
\end{theorem}
\begin{proof}
Assume that \eqref{homogeneousBVP} has only the trivial solution. Then, for $t\in\mathbb{N}_a^b$, we have
\begin{align*}
x(t) &= \int_{a+N}^b G(t,s)h(s)\nabla s \\
&= \int_{a+N}^t v(t,s)h(s)\nabla s + \int_t^b u(t,s)h(s)\nabla s \\
&= \int_{a+N}^t [u(t,s)+x(t,s)]h(s)\nabla s + \int_t^b u(t,s)h(s)\nabla s \\
&= \int_{a+N}^b u(t,s)h(s)\nabla s + \int_{a+N}^t x(t,s)h(s)\nabla s.
\end{align*}

By Theorem \ref{variationOfConstantsThm}, we have that $z(t):= \int_{a+N}^t x(t,s)h(s) \nabla s$ solves the initial value problem
\[
\begin{cases}
L_a^{\nu}z(t)=h(t),\quad t\in\mathbb{N}_{a+N+1}, \\
\nabla^i z(a+i)= 0, \quad i\in\mathbb{N}_0^N.
\end{cases}
\]
Hence, we rewrite $x(t)=\int_{a+N}^b u(t,s)h(s) \nabla s +z(t).$ Applying the operator $L_a^{\nu}$ yields
\begin{align*}
L_a^{\nu} x(t) &= L_a^{\nu}\int_{a+N}^b u(t,s)h(s) \nabla s + L_a^{\nu}z(t) \\
&= \int_{a+N}^b L_a^{\nu}u(t,s)h(s) \nabla s+ h(t) \\
&= \int_{a+N}^b 0\cdot h(s) \nabla s+ h(t) \\
&= h(t).
\end{align*}
Thus, $x(t)=\int_{a+N}^b G(t,s)h(s)\nabla s$ solves the appropriate fractional difference equation.

Consider the left boundary conditions; i.e. for fixed $i\in\bbN_0^{N-1}$,
\begin{align*}
\sum_{j=0}^N \alpha_{ij}\nabla^j x(a+j) &= \sum_{j=0}^N \alpha_{ij}\left[\nabla^j \left. \int_{a+N}^b u(t,s)h(s) \nabla s \right|_{t=a+j}+\nabla^j z(a+j)\right] \\
&= \int_{a+N}^b \left[\sum_{j=0}^N \alpha_{ij}\left.\nabla_t^j u(t,s)\right|_{t=a+j}\right]h(s)\nabla s+\sum_{j=0}^N \alpha_{ij}\nabla^j z(a+j) \\
&= \int_{a+N}^b 0 \cdot h(s)\nabla s+\sum_{j=0}^N a_{ij} \cdot 0 \\
&= 0.
\end{align*}
Therefore, the $N$ left boundary conditions are satisfied. Now consider the right boundary condition
\begin{align*}
\sum_{j=0}^N \beta_{j}\nabla^j x(b) &= \sum_{j=0}^N \beta_{j}\left[\nabla^j\left( \int_{a+N}^b u(b,s)h(s) \nabla s + \int_{a+N}^b x(b,s) h(s) \nabla s \right)\right] \\
&= \int_{a+N}^b \left[\sum_{j=0}^N \beta_{j} \nabla^j u(b,s) \right]h(s) \nabla s + \int_{a+N}^b \left[\sum_{j=0}^N \beta_{j} \nabla^j x(b,s) \right]h(s) \nabla s \\
&= 0,
\end{align*}
using the right boundary condition in the Green's function definition. This shows the right boundary condition is satisfied, so $x(t)=\int_{a+N}^b G(t,s)h(s) \nabla s$ solves \eqref{nonHomogeneousBVP} where $A_i=B=0$, for $i\in\bbN_0^{N-1}$.
\end{proof}

\section{Green's Function for a (2,1) Conjugate Boundary Value Problem}
We will now focus on a specific case of the previous section's boundary value problem. Here, we consider the operator $L_a^\nu$ where $p(t)\equiv 1$, $q(t)\equiv 0$, and $1 < \nu < 2$; i.e. the fractional operator we will consider is $\nabla \nabla_{a^*}^\nu$, and we will use (2,1) conjugate boundary conditions. The homogeneous BVP is
\begin{equation}\label{conjugateHomogeneousBVP}
\begin{cases}
\nabla \nabla_{a^*}^\nu x(t) = 0, \quad t \in \bbN_{a+3}^b, \\
x(a) = 0, \\
\nabla x(a+1) = 0,\\
x(b) = 0,
\end{cases}
\end{equation}
and the corresponding nonhomogeneous BVP is
\begin{equation}\label{conjugateNonHomogeneousBVP}
\begin{cases}
\nabla \nabla_{a^*}^\nu x(t) = h(t), \quad t \in \bbN_{a+3}^b, \\
x(a) = A, \\
\nabla x(a+1) = B,\\
x(b) = C,
\end{cases}
\end{equation}
for $A$, $B$, $C \in \bbR$ and $h: \bbN_{a+3}^b \to \bbR$.

\begin{theorem}
The Green's function for \eqref{conjugateHomogeneousBVP} is given by

\begin{equation}\label{conjugateGreensFunction}
G(t,s) = \begin{cases}
-H_\nu(b,\rho(s)) \frac{t-a-H_\nu(t,a)}{b-a-H_\nu(b,a)}, & t \in \bbN_a^{b-2}\text{ and } s \in \bbN_{\max\left(t+1,a+3\right)}^b, \\
-H_\nu(b,\rho(s)) \frac{t-a-H_\nu(t,a)}{b-a-H_\nu(b,a)} +H_\nu(t,\rho(s)), & t \in \bbN_{a+2}^b \text{ and } s \in \bbN_{a+3}^{\min\left(t+1,b\right)}.
\end{cases}
\end{equation}
\end{theorem}
\begin{proof}
It can be shown that these boundary conditions imply that \eqref{conjugateHomogeneousBVP} has only the trivial solution by recognizing three linearly independent solutions to $\nabla \nabla_{a^*}^\nu x(t) =0$ are given by $x_1(t) = 1$, $x_2(t)= t^{\overline{1}}$, and $x_3(t) = \nabla_a^{-\nu} 1$. So, we consider BVP \eqref{uGreensBVP} in Definition \ref{greensFnctnDefn}.

A general solution to \eqref{uGreensBVP} is given by $u(t,s) = c_1(s) + c_2(s) t + c_3(s) \nabla_a^{-\nu}1$. Since $(\nabla_a^{-\nu}1) (t) = H_\nu(t,a)$, we have that $u(t,s) = c_1(s) + c_2(s)t + c_3(s) H_\nu(t,a)$. This implies $\nabla_t u(t,s) = c_2(s) + c_3(s) H_{\nu-1}(t,a)$. Applying the first, second, and third boundary conditions, we have
\begin{align*}
c_1(s) + c_2(s)a &= 0 \\
c_2(s) + c_3(s) &= 0 \\
c_1(s) + c_2(s)b + c_3(s) H_\nu(b,a) &= -H_\nu(b, \rho(s)),
\end{align*}
noting that in the third boundary condition we use Example \ref{mostBasicCauchyFnctn}, where the Cauchy function is given by $x(t,s) = H_\nu(t,\rho(s))$.
Solving this linear system of equations, we find that
\[
\left( \begin{matrix} c_1(s) \\ c_2(s) \\ c_3(s) \end{matrix} \right) = - \frac{H_\nu(b,\rho(s))}{b-a-H_\nu(b,a)} \left( \begin{matrix} -a \\ 1 \\ -1 \end{matrix} \right),
\]
implying
\[
u(t,s) = - H_\nu(b,\rho(s))\frac{t-a-H_\nu(t,a)}{b-a-H_\nu(b,a)}.
\]
\end{proof}
Since the Cauchy function for $\nabla \nabla_{a^*}^\nu x(t)=0$ is $H_\nu(t,\rho(s))$, we have $v(t,s) = u(t,s) + H_\nu(t,\rho(s))$.

\end{document}